\documentclass[12pt]{article}
\usepackage{amssymb}
\usepackage{amsmath}
\usepackage{amsthm}
\usepackage{comment}
\usepackage{graphicx}

\begin{document}

\def\cM{\mathcal{M}}

\newtheorem{theorem}{Theorem}[section]
\newtheorem{lemma}[theorem]{Lemma}
\newtheorem{proposition}[theorem]{Proposition}





\title{On the relative coexistence of fixed points and period-two solutions near border-collision bifurcations.}
\author{D.J.W.~Simpson\\\\
Institute of Fundamental Sciences\\
Massey University\\
Palmerston North\\
New Zealand}
\maketitle

\begin{abstract}

At a border-collision bifurcation a fixed point of a piecewise-smooth map
intersects a surface where the functional form of the map changes.
Near a generic border-collision bifurcation
there are two fixed points,
each of which exists on one side of the bifurcation.
A simple eigenvalue condition indicates whether the fixed points exist
on different sides of the bifurcation (this case can be interpreted as the persistence of a single fixed point),
or on the same side of the bifurcation (in which case the bifurcation is akin to a saddle-node bifurcation).
A similar eigenvalue condition indicates whether or not there exists a period-two solution on one side of the bifurcation.
Previously these conditions have been combined to
obtain five distinct scenarios
for the existence and relative coexistence of fixed points and period-two solutions near border-collision bifurcations.
In this Letter, it is shown that one of these scenarios,
namely that two fixed points exist on one side of the bifurcation
and a period-two solution exists on the other side of the bifurcation, cannot occur.
The remaining four scenarios are feasible.
Therefore there are exactly four distinct scenarios
for fixed points and period-two solutions near border-collision bifurcations.

\end{abstract}





\section{Introduction}
\label{sec:intro}
\setcounter{equation}{0}


A piecewise-smooth map on $\cM \subset \mathbb{R}^N$ is a discrete-time dynamical system
\begin{equation}
X_{i+1} = F^j(X_i) \;, \qquad X_i \in \cM_j \;,
\label{eq:genPWSMap}
\end{equation}
where the regions $\cM_j$ form a partition of the domain $\cM$, and each $F^j : \cM_j \to \cM$ is a smooth function.
The boundaries of the $\cM_j$, termed switching manifolds, are assumed to be either smooth or piecewise-smooth surfaces.
Piecewise-smooth maps are used to model oscillatory dynamics in systems involving abrupt events,
such as mechanical systems with impacts \cite{LeNi04}, 
power electronics with switching events \cite{ZhMo03},
and economics systems with non-negativity conditions or optimisation \cite{PuSu06}.

As parameters are varied,
a fixed point of (\ref{eq:genPWSMap}) may intersect a switching manifold.
If, near the intersection, the switching manifold is smooth,
(\ref{eq:genPWSMap}) is continuous,
and the derivatives ${\rm D}_X F^j$ are bounded,
then the intersection is known as a {\em border-collision bifurcation} \cite{DiBu08}. 
Dynamics near a border-collision bifurcation of (\ref{eq:genPWSMap})
are well-approximated by a piecewise-linear, continuous map,
which can be put in the form 
\begin{equation}
x_{i+1} = \left\{ \begin{array}{lc}
A_L x_i + b \mu \;, & s_i \le 0 \\
A_R x_i + b \mu \;, & s_i \ge 0
\end{array} \right. \;,
\label{eq:f}
\end{equation}
where, throughout this Letter, $s = e_1^{\sf T} x$
denotes the first component of $x \in \mathbb{R}^N$.
In (\ref{eq:f}),
$A_L$ and $A_R$ are real-valued $N \times N$ matrices,
$b \in \mathbb{R}^N$,
and $\mu \in \mathbb{R}$ is the primary bifurcation parameter:
the border-collision bifurcation occurs at $x=0$ when $\mu=0$.
The requirement that (\ref{eq:f}) is continuous implies
\begin{equation}
A_R = A_L + \xi e_1^{\sf T} \;,
\label{eq:xi}
\end{equation}
for some $\xi \in \mathbb{R}^N$.

A fixed point of (\ref{eq:f}) must be a fixed point of one of the two half-maps of (\ref{eq:f}):
\begin{equation}
f^L(x_i) = A_L x_i + b \mu \;, \qquad
f^R(x_i) = A_R x_i + b \mu \;.
\label{eq:fJ}
\end{equation}
As long as $1$ is not an eigenvalue of $A_L$ and $A_R$, $f^L$ and $f^R$ have unique fixed points, 
\begin{equation}
x^L = (I-A_L)^{-1} b \mu \;, \qquad
x^R = (I-A_R)^{-1} b \mu \;.
\label{eq:xJ}
\end{equation}
The point $x^L$ is a fixed point of (\ref{eq:f}), and said to be admissible, if $s^L \le 0$. 
Similarly, $x^R$ is admissible if $s^R \ge 0$.
Since $x^L$ and $x^R$ are linear functions of $\mu$,
generically $x^L$ and $x^R$ are each admissible for exactly one sign of $\mu$.
In general, for the purposes of characterising the behaviour of (\ref{eq:f}),
it suffices to consider only the sign of $\mu$,
because the structure of the dynamics of (\ref{eq:f}) is independent to the magnitude of $\mu$.

Other invariant sets may be created in border-collision bifurcations,
such as periodic solutions, invariant circles, and chaotic sets \cite{DiBu08,NuYo92,BaGr99,ZhMo06b,SuGa08,Si10},
as well as exotic dynamics such as multi-dimensional attractors \cite{GlWo11},
and infinitely many coexisting attractors \cite{Si14}.
This Letter concerns only fixed points and period-two solutions.
Period-two solutions were first explored by Mark Feigin in the 1970's \cite{Fe70,Fe78},
and were described more recently in \cite{DiBu08,DiFe99}.
The creation of a period-two solution in a border-collision bifurcation
has different scaling properties than a period-doubling bifurcation,
and such differences can have important physical interpretations \cite{ZhSc07}.

In generic situations, (\ref{eq:f}) either has no period-two solution for either sign of $\mu$,
or has an $LR$-cycle (a period-two solution consisting of one point on each side of $s=0$)
for exactly one sign of $\mu$ \cite{Fe70}.
In \cite{Fe78}, Feigin showed
that the relative coexistence of the fixed points $x^L$ and $x^R$
is determined by a simple condition on the eigenvalues of $A_L$ and $A_R$,
and that a similar condition indicates
whether or not an $LR$-cycle exists for one sign of $\mu$.
This is one of the most far-reaching results in the bifurcation theory of nonsmooth dynamical systems,
because it applies to maps of any number of dimensions.
Centre manifold analysis, which is the key tool for dimension reduction,
requires local differentiability and so usually cannot be applied to bifurcations specific to nonsmooth dynamical systems,
such as border-collision bifurcations \cite{GlJe12}.

By directly combining the two generic cases for the nature of both fixed points and period-two solutions,
it appears that border-collision bifurcations can be categorised into five basic scenarios.
In the absence of an $LR$-cycle there are two scenarios:
either $x^L$ and $x^R$ are admissible for different signs of $\mu$, Fig.~\ref{fig:fixedPointsTwoCycles}-A,
or $x^L$ and $x^R$ are admissible for the same sign of $\mu$, Fig.~\ref{fig:fixedPointsTwoCycles}-B.
If there exists an $LR$-cycle, and $x^L$ and $x^R$ are admissible for different signs of $\mu$,
then, trivially, the $LR$-cycle coexists with exactly one fixed point, Fig.~\ref{fig:fixedPointsTwoCycles}-C.
Finally, if there exists an $LR$-cycle, and $x^L$ and $x^R$ are admissible for the same sign of $\mu$,
it appears that there are two scenarios.
The $LR$-cycle could either coexist with $x^L$ and $x^R$, as in Fig.~\ref{fig:fixedPointsTwoCycles}-D,
or coexist with neither $x^L$ or $x^R$.
In \cite{Fe78}, Feigin noted that the latter scenario is not possible
in one-dimension ($N=1$) in view of Sharkovskii's theorem \cite{Sh95}.
Feigin further stated that this scenario is not possible for $N=2$ (but did not provide a proof),
and conjectured that the scenario is not possible for any $N \in \mathbb{Z}^+$.
The purpose of this Letter is to prove this conjecture.

Each of the four scenarios of Fig.~\ref{fig:fixedPointsTwoCycles}
is possible for (\ref{eq:f}) in any number of dimensions.
In Fig.~\ref{fig:fixedPointsTwoCycles} the scenarios are illustrated for (\ref{eq:f}) with $N=1$, for which (\ref{eq:f}) is written as
\begin{equation}
x_{i+1} = \left\{ \begin{array}{lc}
a_L x_i + \mu \;, & x_i \le 0 \\
a_R x_i + \mu \;, & x_i \ge 0
\end{array} \right. \;,
\label{eq:f1d}
\end{equation}
where $a_L, a_R \in \mathbb{R}$.

The remainder of this Letter is organised as follows.
Calculations for fixed points and period-two solutions of (\ref{eq:f})
are given in \S\ref{sec:fixedPoints} and \S\ref{sec:periodTwo}, respectively.
The basic border-collision bifurcation scenarios
formed by considering all generic possibilities for fixed points and
period-two solutions are described in \S\ref{sec:classification}.
In \S\ref{sec:impossibility} it is proved that
a non-degenerate period-two solution of (\ref{eq:f}) must coexist with a fixed point.
Finally, \S\ref{sec:conc} presents a brief summary and outlook. 

\begin{figure}[t!]
\begin{center}
\setlength{\unitlength}{1cm}
\begin{picture}(15,10)
\put(0,0){\includegraphics[height=10cm]{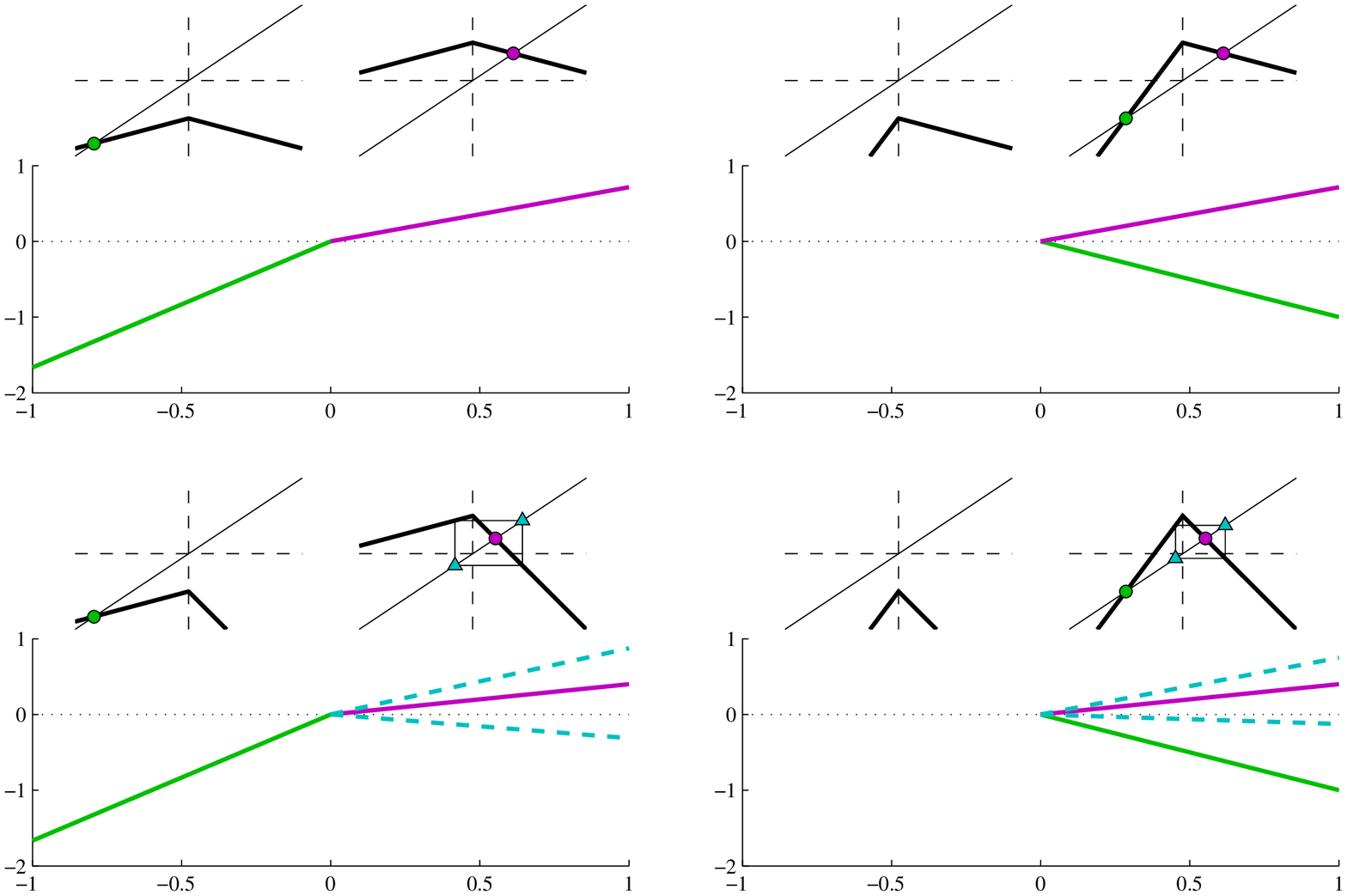}}
\put(.8,9.5){\large \sf \bfseries A}
\put(3.77,5){\small $\mu$}
\put(.3,6.66){\small $x$}
\put(2.07,9.56){\footnotesize $\mu<0$}
\put(5.09,9.56){\footnotesize $\mu>0$}
\put(8.3,9.5){\large \sf \bfseries B}
\put(11.27,5){\small $\mu$}
\put(7.8,6.66){\small $x$}
\put(9.57,9.56){\footnotesize $\mu<0$}
\put(12.59,9.56){\footnotesize $\mu>0$}
\put(.8,4.5){\large \sf \bfseries C}
\put(3.77,0){\small $\mu$}
\put(.3,1.66){\small $x$}
\put(2.07,4.56){\footnotesize $\mu<0$}
\put(5.09,4.56){\footnotesize $\mu>0$}
\put(8.3,4.5){\large \sf \bfseries D}
\put(11.27,0){\small $\mu$}
\put(7.8,1.66){\small $x$}
\put(9.57,4.56){\footnotesize $\mu<0$}
\put(12.59,4.56){\footnotesize $\mu>0$}
\end{picture}
\caption{
Bifurcation diagrams
showing the four scenarios for the existence and relative coexistence of fixed points
and period-two solutions near generic border-collision bifurcations of piecewise-smooth maps.
In panels A and C, the two fixed points $x^L$ and $x^R$ (\ref{eq:xJ})
are admissible for different signs of $\mu$ (persistence);
in panels B and D, $x^L$ and $x^R$ are admissible for the same sign of $\mu$ (nonsmooth-fold).
In panels A and B, there is no period-two solution;
in panels C and D, an $LR$-cycle $\{ x^{LR} ,\, x^{RL} \}$ exists for one sign of $\mu$ (indicated by dashed lines).
By Theorem \ref{th:caseFive}, the $LR$-cycle must coexist with at least one fixed point.
Each bifurcation diagram is illustrated for the one-dimensional map (\ref{eq:f1d})
(the insets are graphs of $x_{i+1}$ versus $x_i$).
The specific parameter values are
$(a_L,a_R) = (0.4,-0.4)$ in panel A,
$(a_L,a_R) = (2,-0.4)$ in panel B,
$(a_L,a_R) = (0.4,-1.5)$ in panel C,
and $(a_L,a_R) = (2,-1.5)$ in panel D.
\label{fig:fixedPointsTwoCycles}
}
\end{center}
\end{figure}

\section{Fixed points}
\label{sec:fixedPoints}
\setcounter{equation}{0}

In order to compare the values of $s^L$ and $s^R$
(the first components of $x^L$ and $x^R$ (\ref{eq:xJ})), we let
\begin{equation}
\varrho^{\sf T} = e_1^{\sf T} {\rm adj}(I-A_L) \;,
\label{eq:varrho}
\end{equation}
where ${\rm adj}(A)$ denotes the {\em adjugate} of a square matrix $A$.
Recall, if $A$ is nonsingular, then $A^{-1} = \frac{{\rm adj}(A)}{\det(A)}$.
Thus, by (\ref{eq:xJ}) with $J=L$, we have
\begin{equation}
s^L = \frac{\varrho^{\sf T} b}{\det(I-A_L)} \mu \;.
\label{eq:sL}
\end{equation}
Since $A_L$ and $A_R$ differ in only their first columns (\ref{eq:xi}),
${\rm adj}(I-A_L)$ and ${\rm adj}(I-A_R)$ have the same first row \cite{Si10,DiFe99},
that is, $e_1^{\sf T} {\rm adj}(I-A_R) = \varrho^{\sf T}$.
Thus, by (\ref{eq:xJ}) with $J=R$, we have
\begin{equation}
s^R = \frac{\varrho^{\sf T} b}{\det(I-A_R)} \mu \;.
\label{eq:sR}
\end{equation}
By (\ref{eq:sL}) and (\ref{eq:sR}), if $\varrho^{\sf T} b = 0$, then $s^L = s^R = 0$ for all $\mu$.
In this instance the fixed points do not move away from the switching manifold as $\mu$ is varied from zero,
which runs counter to our notion of a border-collision bifurcation.
For this reason, $\varrho^{\sf T} b \ne 0$ is a non-degeneracy condition for the border-collision bifurcation of (\ref{eq:f}) at $\mu=0$.

Following \cite{Fe78}, for $J = L,R$,
we let $\sigma_J^+$ denote the number of real eigenvalues of $A_J$ that are greater than $1$.
If $1$ is not an eigenvalue of $A_J$, then
\begin{equation}
{\rm sgn} \left( \det(I-A_J) \right) = (-1)^{\sigma_J^+} \;, \quad J = L,R \;.
\label{eq:sigmaJplus}
\end{equation}
By combining (\ref{eq:sL}), (\ref{eq:sR}) and (\ref{eq:sigmaJplus}), we obtain the formula
\begin{equation}
{\rm sgn} \left( s^J \right) = (-1)^{\sigma_J^+} \,{\rm sgn} \left( \varrho^{\sf T} b \mu \right) \;, \quad J = L,R \;.
\label{eq:sgnsJ}
\end{equation}
Recall, $x^L$ is admissible if $s^L \le 0$, and $x^R$ is admissible if $s^R \ge 0$.
Therefore, if $\sigma_L^+ + \sigma_R^+$ is an even number,
then $(-1)^{\sigma_L^+} = (-1)^{\sigma_R^+}$,
and hence by (\ref{eq:sgnsJ}), ${\rm sgn} \left( s^L \right) = {\rm sgn} \left( s^R \right)$.
Thus in this case $x^L$ and $x^R$ are admissible for different signs of $\mu$ ({\em persistence} of a fixed point). 
Alternatively if $\sigma_L^+ + \sigma_R^+$ is an odd number,
then $x^L$ and $x^R$ are admissible for the same sign of $\mu$ (a {\em nonsmooth-fold}). 


\section{Period-two solutions}
\label{sec:periodTwo}
\setcounter{equation}{0}

Let us first consider a period-two solution of (\ref{eq:f}) consisting of two points with $s < 0$.
Points of this solution are fixed points of
$\left( f^L \circ f^L \right)(x_i) = A_L^2 x_i + (I+A_L) b \mu$.
If $A_L$ does not have an eigenvalue of $1$ or $-1$, as is generically the case,
this period-two solution is unique, and therefore must coincide with $x^L$.
Hence the period-two solution is really a fixed point, and we do not need to consider it further.
We can similarly dismiss period-two solutions of (\ref{eq:f}) consisting of two points with $s > 0$.
Therefore it remains to consider an $LR$-cycle $\{ x^{LR} ,\, x^{RL} \}$,
where $x^{RL} = f^L(x^{LR})$ and $x^{LR} = f^R(x^{RL})$.

Expressions for the first components of $x^{LR}$ and $x^{RL}$ are given by the following lemma.
Lemma \ref{le:sLRsRL} is a special case of a result for general periodic solutions of (\ref{eq:f}) derived in \cite{Si10,SiMe10},
and the reader is referred to these sources for a proof.

\begin{lemma}
If $1$ is not an eigenvalue of $A_R A_L$, then the $LR$-cycle is unique
(but not necessarily admissible) and
\begin{equation}
s^{LR} = \frac{\det(I+A_R) \varrho^{\sf T} b}{\det(I-A_R A_L)} \mu \;, \qquad
s^{RL} = \frac{\det(I+A_L) \varrho^{\sf T} b}{\det(I-A_R A_L)} \mu \;.
\label{eq:sLRsRL}
\end{equation}
\label{le:sLRsRL}
\end{lemma}

As in \cite{Fe78}, we let $\sigma_J^-$ denote the number of real eigenvalues of $A_J$ that are less than $-1$.
If $-1$ is not an eigenvalue of $A_J$, then
\begin{equation}
{\rm sgn} \left( \det(I+A_J) \right) = (-1)^{\sigma_J^-} \;, \quad J = L,R \;.
\label{eq:sigmaJminus}
\end{equation}
We also let $\sigma_{LR}^+$ denote the number of real eigenvalues of $A_R A_L$
(or equivalently of $A_L A_R$) that are greater than $1$.
If $1$ is not an eigenvalue of $A_R A_L$, then
\begin{equation}
{\rm sgn} \left( \det(I-A_R A_L) \right) = (-1)^{\sigma_{LR}^+} \;.
\label{eq:sigmaLRplus}
\end{equation}
By (\ref{eq:sLRsRL}), (\ref{eq:sigmaJminus}) and (\ref{eq:sigmaLRplus}), we have
\begin{equation}
{\rm sgn} \left( s^{LR} \right) = (-1)^{\sigma_R^- + \sigma_{LR}^+}
\,{\rm sgn} \left( \varrho^{\sf T} b \mu \right) \;, \qquad
{\rm sgn} \left( s^{RL} \right) = (-1)^{\sigma_L^- + \sigma_{LR}^+}
\,{\rm sgn} \left( \varrho^{\sf T} b \mu \right) \;,
\label{eq:sgnsLRsRL}
\end{equation}
The $LR$-cycle is admissible if $s^{LR} \le 0$ and $s^{RL} \ge 0$.
Therefore by (\ref{eq:sgnsLRsRL}), if $\sigma_L^- + \sigma_R^-$ is even,
then ${\rm sgn} \left( s^{LR} \right) = {\rm sgn} \left( s^{RL} \right)$
and so the $LR$-cycle is not admissible for all $\mu \ne 0$.
Alternatively if $\sigma_L^- + \sigma_R^-$ is odd,
then the $LR$-cycle is admissible for one sign of $\mu$.


\section{Feigin's classification}
\label{sec:classification}
\setcounter{equation}{0}

If the $LR$-cycle is admissible for one sign of $\mu$,
we would like to determine which fixed points it coexists with.
To this end, we let $\sigma_{LL}^+$ denote the number of real eigenvalues of $A_L^2$ that are greater than $1$.
If $1$ is not an eigenvalue of $A_L^2$, then
\begin{equation}
{\rm sgn} \left( \det \left( I-A_L^2 \right) \right) = (-1)^{\sigma_{LL}^+} \;.
\label{eq:sigmaLLplus}
\end{equation}
In view of the simple factorisation $I-A_L^2 = (I-A_L)(I+A_L)$,
by (\ref{eq:sigmaJplus}), (\ref{eq:sigmaJminus}) and (\ref{eq:sigmaLLplus}) we have
\begin{equation}
(-1)^{\sigma_{LL}^+} = (-1)^{\sigma_L^+ + \sigma_L^-} \;.
\label{eq:sigmaLLplusFormula}
\end{equation}

The following theorem summarises the main results of \cite{Fe78}.
All aspects of Theorem \ref{th:FeiginAll} follow from the results of the previous two sections,
except those relating to the quantity $\sigma_{LL}^+ + \sigma_{LR}^+$,
and for a complete proof the reader is referred to \cite{DiBu08,Fe78,DiFe99}.

\begin{theorem}
For the map (\ref{eq:f}), suppose
$\varrho^{\sf T} b \ne 0$,
$1$ is not an eigenvalue of $A_L$, $A_R$ and $A_R A_L$,
and $-1$ is not an eigenvalue of $A_L$ and $A_R$.
\begin{enumerate}
\item
If $\sigma_L^+ + \sigma_R^+$ and $\sigma_L^- + \sigma_R^-$ are even,
then $x^L$ and $x^R$ are admissible for different signs of $\mu$,
and the $LR$-cycle is not admissible for all $\mu \ne 0$.
\item
If $\sigma_L^+ + \sigma_R^+$ is odd and $\sigma_L^- + \sigma_R^-$ is even,
then $x^L$ and $x^R$ are admissible for the same sign of $\mu$,
and the $LR$-cycle is not admissible for all $\mu \ne 0$.
\item
If $\sigma_L^+ + \sigma_R^+$ is even and $\sigma_L^- + \sigma_R^-$ is odd,
then $x^L$ and $x^R$ are admissible for different signs of $\mu$,
and the $LR$-cycle is admissible for one sign of $\mu$.
If $\sigma_{LL}^+ + \sigma_{LR}^+$ is even [odd],
then the $LR$-cycle coexists with $x^R$ [$x^L$].
\item
If $\sigma_L^+ + \sigma_R^+$, $\sigma_L^- + \sigma_R^-$ and $\sigma_{LL}^+ + \sigma_{LR}^+$ are odd,
then $x^L$, $x^R$ and the $LR$-cycle are admissible for the same sign of $\mu$.
\item
If $\sigma_L^+ + \sigma_R^+$ and $\sigma_L^- + \sigma_R^-$ are odd
and $\sigma_{LL}^+ + \sigma_{LR}^+$ is even,
then $x^L$ and $x^R$ are admissible for one sign of $\mu$,
and the $LR$-cycle is admissible for the other sign of $\mu$.
\end{enumerate}
\label{th:FeiginAll}
\end{theorem}



\section{The $LR$-cycle coexists with at least one fixed point}
\label{sec:impossibility}
\setcounter{equation}{0}

As a consequence of the following theorem, which is the main result of this Letter,
if $\sigma_L^+ + \sigma_R^+$ and $\sigma_L^- + \sigma_R^-$ are odd,
then $\sigma_{LL}^+ + \sigma_{LR}^+$ is also odd.
Therefore, scenario (v) of Theorem \ref{th:FeiginAll} cannot occur.

\begin{theorem}
Suppose $1$ is not an eigenvalue of $A_L$, $A_R$ and $A_R A_L$,
and $-1$ is not an eigenvalue of $A_L$ and $A_R$.
Suppose $(-1)^{\sigma_L^+ + \sigma_R^+} = (-1)^{\sigma_L^- + \sigma_R^-}$.
Then $(-1)^{\sigma_{LL}^+ + \sigma_{LR}^+} = (-1)^{\sigma_L^+ + \sigma_R^+}$.
\label{th:caseFive}
\end{theorem}

The key feature of the proof of Theorem \ref{th:caseFive}, given below,
is that we look closely at $(-1)^{\sigma_L^+ + \sigma_R^-}$ and $(-1)^{\sigma_L^- + \sigma_R^+}$,
rather than $(-1)^{\sigma_L^+ + \sigma_R^+}$ and $(-1)^{\sigma_L^- + \sigma_R^-}$,
as the first two quantities admit a convenient algebraic manipulation.
We begin with the following lemma.

\begin{lemma}
Suppose $1$ and $-1$ are not eigenvalues of $A_L$ and $A_R$.
Then
\begin{align}
(-1)^{\sigma_L^+ + \sigma_R^-} &=
{\rm sgn} \left( \det(I-A_R A_L) +
e_1^{\sf T} {\rm adj}(I-A_R A_L) \xi \right) \;, \label{eq:crossFormulaPlus} \\
(-1)^{\sigma_L^- + \sigma_R^+} &=
{\rm sgn} \left( \det(I-A_R A_L) -
e_1^{\sf T} {\rm adj}(I-A_R A_L) \xi \right) \;, \label{eq:crossFormulaMinus}
\end{align}
where $\xi$ is given by (\ref{eq:xi}). 
\label{le:crossFormula}
\end{lemma}

\begin{proof}[Proof of Lemma \ref{le:crossFormula}]
For clarity we derive only (\ref{eq:crossFormulaPlus}).
Equation (\ref{eq:crossFormulaMinus}) results from switching signs in the following arguments.
By (\ref{eq:sigmaJplus}) and (\ref{eq:sigmaJminus}),
\begin{equation}
(-1)^{\sigma_L^+ + \sigma_R^-} = {\rm sgn} \left( \det(I+A_R) \det(I-A_L) \right) \;.
\label{eq:crossFproof1}
\end{equation}
We can use (\ref{eq:xi}) to write
\begin{equation}
(I+A_R)(I-A_L) = I - A_R A_L + \xi e_1^{\sf T} \;.
\label{eq:crossFproof2}
\end{equation}
Next we recall the {\em matrix determinant lemma} \cite{Be05}:
$\det(A + p q^{\sf T}) \equiv \det(A) + q^{\sf T} {\rm adj}(A) p$,
for any $N \times N$ matrix $A$, and $p,q \in \mathbb{R}^N$.
By applying this result to the right-hand side of (\ref{eq:crossFproof2}),
from (\ref{eq:crossFproof1}) we obtain (\ref{eq:crossFormulaPlus}).
\end{proof}

\begin{proof}[Proof of Theorem \ref{th:caseFive}]
By assumption $(-1)^{\sigma_L^+ + \sigma_R^+} = (-1)^{\sigma_L^- + \sigma_R^-}$,
therefore $(-1)^{\sigma_L^+ + \sigma_R^+ + \sigma_L^- + \sigma_R^-} = 1$,
and therefore we have $(-1)^{\sigma_L^+ + \sigma_R^-} = (-1)^{\sigma_L^- + \sigma_R^+}$.
By Lemma \ref{le:crossFormula}, we have
$\left| \det(I-A_R A_L) \right| > \left| e_1^{\sf T} {\rm adj}(I-A_R A_L) \xi \right|$,
and ${\rm sgn} \left( \det(I-A_R A_L) \right) = (-1)^{\sigma_L^- + \sigma_R^+}$.
By (\ref{eq:sigmaLRplus}) we can rewrite this last equation as
\begin{equation}
(-1)^{\sigma_{LR}^+} = (-1)^{\sigma_L^- + \sigma_R^+} \;.
\label{eq:caseFproof1}
\end{equation}
By multiplying (\ref{eq:sigmaLLplusFormula}) and (\ref{eq:caseFproof1}) together, we obtain
$(-1)^{\sigma_{LL}^+ + \sigma_{LR}^+} =
(-1)^{\sigma_L^+ + 2 \sigma_L^- + \sigma_R^+} =
(-1)^{\sigma_L^+ + \sigma_R^+}$, as required.
\end{proof}

\section{Discussion}
\label{sec:conc}
\setcounter{equation}{0}

By Theorems \ref{th:FeiginAll} and \ref{th:caseFive},
the existence and relative coexistence of $x^L$, $x^R$ and the $LR$-cycle 
near generic border-collision bifurcations is almost completely determined by the
even/odd parity of $\sigma_L^+ + \sigma_R^+$ and $\sigma_L^- + \sigma_R^-$.
We only need to evaluate $\sigma_{LL}^+ + \sigma_{LR}^+$
if we wish to identify which fixed point the $LR$-cycle coexists with in scenario (iii) of Theorem \ref{th:FeiginAll}.
Scenario (v) of Theorem \ref{th:FeiginAll} cannot occur in view of Theorem \ref{th:caseFive},
which was proved by using algebraic arguments to demonstrate that the
particular combination of eigenvalue conditions required for scenario (v) cannot be satisfied.

The stability of $x^L$, $x^R$ and the $LR$-cycle was not discussed here, refer to \cite{DiBu08,Fe78,DiFe99}.
In brief, $x^L$, $x^R$ and the $LR$-cycle are attracting if and only if
all eigenvalues of $A_L$, $A_R$ and $A_R A_L$, respectively, have modulus less than $1$.
Stability therefore relates directly to the various $\sigma$'s defined above,
and Theorem \ref{th:FeiginAll} can be used to show
that for any $\mu \in \mathbb{R}$, at most one fixed point or period-two solution can be attracting.

The admissibility of periodic solutions of (\ref{eq:f}) with period greater than two
cannot be characterised as simply as for fixed points and period-two solutions.
For instance, a generic $LLR$-cycle 
is admissible for one sign of $\mu$ if and only if
\begin{equation}
{\rm sgn} \left( \det(I + A_R + A_R A_L) \right) =
{\rm sgn} \left( \det(I + A_L + A_L A_R) \right) \ne
{\rm sgn} \left( \det \left( I + A_L + A_L^2 \right) \right) \;,
\label{eq:llr}
\end{equation}
see \cite{Si10,SiMe10},
and it is not clear how to relate the quantities in (\ref{eq:llr}) to the eigenvalues of $A_L$ and $A_R$.


\end{document}